\documentclass[12pt]{article}

\usepackage{amsmath,amsfonts,amssymb,amsthm}
\input amssym.def
\newcommand{\Z}{{\mathbb Z}}
\newcommand{\C}{{\mathbb C}}

\newcommand{\N}{{\mathbb N}}

\def\<{\langle}
\def\>{\rangle}

\newtheorem{thm}{Theorem}[section]
\newtheorem{prop}[thm]{Proposition}
\newtheorem{lem}[thm]{Lemma}

\newtheorem{rmk}[thm]{Remark}
\newtheorem{definition}[thm]{Definition}

\input amssym.def
\begin{document}
	
	\begin{center}
		{\Large \bf  Automorphism group and twisted modules of the twisted Heisenberg-Virasoro vertex operator algebra}
	\end{center}
	
	\begin{center}
		{Hongyan Guo
			\footnote{Partially supported by
				NSFC (No.11901224)
				and NSF of Hubei Province (No.2019CFB160)
			}
			\\
			School of Mathematics and Statistics,
			Central China Normal University, \\ Wuhan 430079, China
		}
		
	\end{center}
	
	\begin{abstract}
		
		We first determine the automorphism group of the twisted Heisenberg-Virasoro vertex operator algebra $V_{\mathcal{L}}(\ell_{123},0)$.
		Then, for any integer $t>1$, we introduce a new Lie algebra $\mathcal{L}_{t}$, and show that $\sigma_{t}$-twisted $V_{\mathcal{L}}(\ell_{123},0)$($\ell_{2}=0$)-modules are in one-to-one correspondence with restricted $\mathcal{L}_{t}$-modules of level $\ell_{13}$, where $\sigma_{t}$ is an order $t$ automorphism of $V_{\mathcal{L}}(\ell_{123},0)$.
		At the end, we give a complete list of irreducible $\sigma_{t}$-twisted $V_{\mathcal{L}}(\ell_{123},0)$($\ell_{2}=0$)-modules.

	\end{abstract}

	\section{Introduction}
	\def\theequation{1.\arabic{equation}}
	\setcounter{equation}{0}
	
	Let $\mathcal{L}$ be the twisted Heisenberg-Virasoro algebra. It is the universal central extension of the Lie algebra
	of differential operators on a circle of order at most one (cf. \cite{ACKP}):
	$$\{ f(t)\frac{d}{dt}+g(t)\ | \ f(t),g(t)\in\C[t,t^{-1}] \}.$$  $\mathcal{L}$ contains an infinite-dimensional Heisenberg algebra and the Virasoro algebra as subalgebras  (cf. \cite{ACKP}, \cite{B1}). The induced module $V_{\mathcal{L}}(\ell_{123},0)=U(\mathcal{L})\otimes_{U(\mathcal{L}_{(\leq 1)})}\C_{\ell_{123}}$ is a vertex operator algebra of central charge $\ell_{1}$ with conformal vector $\omega=L_{-2}{\bf 1}$ (cf. \cite{GW}). 
	$V_{\mathcal{L}}(\ell_{123},0)$ is a nonrational vertex operator algebra and is not $C_{2}$-cofinite. The structure theory and representation theory of the twisted Heisenberg-Virasoro vertex operator algebra $V_{\mathcal{L}}(\ell_{123},0)$ are closely related to the three scalars $\ell_{1},\ell_{2},\ell_{3}$ (cf. \cite{ACKP}, \cite{AR1}, \cite{AR2}, \cite{B1}, \cite{B2}, \cite{FZ}, \cite{GW}, etc.). 
	
	Determining the automorphism group $\mbox{Aut}(V)$ of a vertex operator algebra $V$ is an important subject in vertex operator algebra theory. It is related to the orbifold theory which studies the fixed point subalgebras of vertex operator algebras and their modules under certain finite subgroups of the full automorphism groups. The orbifold conjecture says that under some conditions on $V$, every simple $V^{G}$-module is contained in some $g$-twisted $V$-module, where $G$ is a finite subgroup of $\mbox{Aut}(V)$, $g\in G$, $V^{G}$ is the fixed point subalgebra of $V$ under the group $G$.

	The automorphism group of the twisted Heisenberg-Virasoro algebra $\mathcal{L}$ has been studied in \cite{SJ}.
	In this paper, we study the automorphism group of the twisted Heisenberg-Virasoro vertex operator algebra $V_{\mathcal{L}}(\ell_{123},0)$.
	We know that $V_{\mathcal{L}}(\ell_{123},0)$
	is generated 
	by $\omega=L_{-2}{\bf 1}$ and $I_{-1}{\bf 1}$. By definition, any homomorphism of a vertex operator algebra $(V,Y,{\bf 1}, \omega)$ takes $\omega$ to $\omega$. Therefore, it suffices to determine the action on $I_{-1}{\bf 1}$.  It turns out that automorphisms of the twisted Heisenberg-Virasoro vertex operator algebra $V_{\mathcal{L}}(\ell_{123},0)$ depend on the numbers $\ell_{2}$ and $\ell_{3}$,
	and there exists automorphism of order other than 2, which makes the study of twisted modules for $V_{\mathcal{L}}(\ell_{123},0)$ interesting.
	
	In \cite{GW}, all irreducible modules for the vertex operator algebra
	$V_{\mathcal{L}}(\ell_{123},0)$ are classified: that is, every irreducible module for $V_{\mathcal{L}}(\ell_{123},0)$ is isomorphic to some $L_{\mathcal{L}}(\ell_{123}, h_{1}, h_{2})$, $h_{1},h_{2}\in\C$.
	Here, for any integer $t>1$, we classify $\sigma_{t}$-twisted irreducible modules for $V_{\mathcal{L}}(\ell_{123},0)$, where $\sigma_{t}$ is an order $t$ automorphism of $V_{\mathcal{L}}(\ell_{123},0)$.
	The way we deal with this problem is similar to the one used for vertex algebra modules (cf. \cite{GLTW}, \cite{GW1}, \cite{GW}, \cite{LL}, etc.), but in the context of twisted modules.
	We first introduce another Lie algebra $\mathcal{L}_{t}$. It is a Lie algebra with basis
	$\{L_{n},I_{n+\frac{1}{t}},k_{1},k_{3}\ | \ n\in\Z  \}$, and Lie brackets
	$$
	[L_{m}, L_{n}]=(m-n)L_{m+n}+\delta_{m+n,0}\frac{m^{3}-m}{12}k_{1},  
	$$
	$$
	[L_{m}, I_{n+\frac{1}{t}}]=-(n+\frac{1}{t}) I_{m+n+\frac{1}{t}},   
	$$
	$$
	[I_{m+\frac{1}{t}}, I_{n+\frac{1}{t}}]=(m+\frac{1}{t})\delta_{m+n+\frac{2}{t},0}\delta_{t,2}k_{3},\;\;\; [\mathcal{L}, k_{i}]=0,\; i=1,3.  
	$$
	Note that when $t\neq 2$, $\{I_{n+\frac{1}{t}} \ | \ n\in\Z  \}$ is an abelian Lie algebra. 
	Then we construct irreducible $\mathcal{L}_{t}$-modules $L_{\mathcal{L}_{t}}(\Bbbk_{1},\Bbbk_{3},h)$ as quotient of the induced modules $M_{\mathcal{L}_{t}}(\Bbbk_{1},\Bbbk_{3},h)$,
	where $\Bbbk_{1},\Bbbk_{3},h\in\mathbb{C}$.
	We show that $\sigma_{t}$-twisted $V_{\mathcal{L}}(\ell_{123},0)$-modules ($\ell_{2}=0$) are in one-to-one correspondence with restricted $\mathcal{L}_{t}$-modules of level $\ell_{13}$.
	Using this result, we get a complete list of irreducible $\sigma_{t}$-twisted
	$V_{\mathcal{L}}(\ell_{123},0)$-modules, where $\ell_{2}=0$, $\ell_{1}$, $\ell_{3}\in\C$. 
	Let $V_{\mathcal{L}}(\ell_{123},0)^{\sigma_{t}}$ be the fixed point subalgebra of $V_{\mathcal{L}}(\ell_{123},0)$ under the automorphism $\sigma_{t}$. $V_{\mathcal{L}}(\ell_{123},0)^{\sigma_{t}}$ is a vertex operator subalgebra of $V_{\mathcal{L}}(\ell_{123},0)$.
	It is important and meaningful to study the representations of the fixed point subalgebra $V_{\mathcal{L}}(\ell_{123},0)^{\sigma_{t}}$.
	We remark at the end of the paper that except for the case of order 2, the complete list of irreducible modules for  $V_{\mathcal{L}}(\ell_{123},0)^{\sigma_{t}}$ needs to be further investigated.
	
	This paper is organized as follows. In Section 2, we review the notions and some results of vertex operator algebras, automorphisms and twisted modules for vertex operator algebras.
	In Section 3, we study the automorphism group of the twisted Heisenberg-Virasoro vertex operator algebra $V_{\mathcal{L}}(\ell_{123},0)$. 
	In Section 4, we first study the twisted modules of $V_{\mathcal{L}}(\ell_{123},0)$($\ell_{2}=0$) under an order $t$ automorphism $\sigma_{t}$ for any integer $t>1$.
	Then we give a complete list of irreducible $\sigma_{t}$-twisted $V_{\mathcal{L}}(\ell_{123},0)$($\ell_{2}=0$)-modules.

	\section{Preliminaries}
	\def\theequation{2.\arabic{equation}}
	\setcounter{equation}{0}
	
	For later use, we recall the following result (cf. Proposition 2.3.7 of \cite{LL}).
	\begin{lem}
		\begin{eqnarray}
		(x_{1}-x_{2})^{m}\Big(\frac{\partial}{\partial x_{2}}\Big)^{n}x_{1}^{-1}
		\delta\Big(\frac{x_{2}}{x_{1}}\Big)=0                     \label{eq:2.1}
		\end{eqnarray}
		for $m>n$, $m,n\in\N$, where $\displaystyle\delta\Big(\frac{x_{1}}{x_{2}}\Big)=\sum\limits_{n\in\Z}x_{1}^{n}x_{2}^{-n}$.
	\end{lem}

	For the definition of vertex (operator) algebra and its modules, we follow \cite{LL}.
	\begin{definition}
		{\em
			A {\em vertex algebra} (V,Y,{\bf 1}) consists of a vector space $V$, a linear map (the vertex operator map)
			$$Y(\cdot, x):V\longrightarrow (End V)[[x,x^{-1}]],\;\;v\mapsto Y(v,x)=\sum_{n\in\Z}v_{n}x^{-n-1},$$
			and a vacuum vector ${\bf 1}$
			such that the following conditions hold for $u,v\in V$:
			\begin{itemize}
				\item [(1)] {\em truncation condition}: $u_{n}v=0$ for $n$ sufficiently large;
				\item[(2)] {\em vacuum property}: $Y({\bf 1},x)=1$;
				\item[(3)] {\em creation property}: $Y(v,x){\bf 1}\in V[[x]]$ and $\lim\limits_{x\rightarrow 0}Y(v,x){\bf 1}=v$;
				\item[(4)] {\em Jacobi identity}: 
				\begin{eqnarray*}
					&&x_{0}^{-1}\delta\left(\frac{x_{1}-x_{2}}{x_{0}}\right)Y(u,x_{1})Y(v,x_{2}) -
					x_{0}^{-1}\delta\left(\frac{x_{2}-x_{1}}{-x_{0}}\right)Y(v,x_{2})Y(u,x_{1})\\
					&&\hspace{2cm}= x_{2}^{-1}\delta\left(\frac{x_{1}-x_{0}}{x_{2}}\right)Y(Y(u,x_{0})v,x_{2}).
				\end{eqnarray*}
			\end{itemize}
		}
	\end{definition}
	
	Let $D$ be the endomorphism of the vertex algebra $V$ defined by 
	$D(v)=v_{-2}{\bf 1}$ for $v\in V$. Then we have $Y(Dv,x)=\displaystyle\frac{d}{dx}Y(v,x)$.
	
	\begin{definition}
		{\em  A {\em vertex operator algebra} is a $\Z$-graded vector space (graded by weights)
			$$V=\coprod_{n\in\Z}V_{(n)}\;\;\mbox{for}\;v\in V_{(n)}, n=wt\;v,$$
			equipped with a vertex algebra structure $(V,Y,{\bf 1})$ and a distinguished homogeneous vector $\omega$ (the {\em conformal vector}) of weight 2 ($\omega\in V_{(2)}$) such that
			\begin{itemize}
				\item [(1)] {\em two grading restrictions:} 
				$$\mbox{dim}\; V_{(n)}<\infty\;\;\;\;\mbox{for}\;n\in\Z,$$
				$$V_{(n)}=0\;\;\;\;\mbox{for}\; n\; \mbox{sufficiently negative};$$
				\item [(2)] {\em Virasoro algebra relations:}
				$$[L(m), L(n)]=(m-n)L(m+n)+\frac{1}{12}(m^{3}-m)\delta_{m+n,0}c_{V}$$
				for $m,n\in\Z$, where
				$$Y(\omega,x)=\sum_{n\in\Z}L(n)x^{-n-2} (=\sum_{n\in\Z}\omega_{n}x^{-n-1});$$
				and $c_{V}\in\C$ ({\em central charge or rank} of $V$);
				\item[(3)]  {\em $L(0)$-grading:}
				$L(0)v=nv=(wt\;v)v$\;\;\;\;for $n\in\Z$ and $v\in V_{(n)}$;
				\item[(4)] {\em $L(-1)$-derivative property:} $Y(L(-1)v,x)=\displaystyle\frac{d}{dx}Y(v,x)$.
			\end{itemize}
			
		}
	\end{definition}
	
	It can be proved that ${\bf 1}\in V_{(0)}$ and $D=L(-1)$.
	
	\begin{definition}
		{\em	An {\em automorphism} of a vertex operator algebra $(V,Y,\omega,{\bf 1})$ is a linear isomorphism $\sigma$ of $V$ 
			such that 
			$$\sigma({\bf 1})={\bf 1},\;\;\sigma(\omega)=\omega \;\;\mbox{and}\;\; \sigma(u_{n}v)=\sigma(u)_{n}\sigma(v)$$
			for any $u,v\in V$, $n\in\mathbb{Z}$.
			
		}
	\end{definition}

	The group of all automorphisms of a vertex operator algebra $V$ is denoted by $\mbox{Aut}(V)$. Any automorphism of a vertex operator algebra $V$ is {\em grading-preserving}, i.e. it preserves each homogeneous subspace $V_{(n)}$ of $V$, $n\in\Z$.
	Let $\sigma$ be an order $t$ automorphism of a vertex operator algebra $V$, $t$ is a positive integer. Then $\sigma$ acts semisimply on $V$. Therefore
	$$V=V^{0}\oplus V^{1}\oplus \cdots\oplus V^{t-1}$$
	where $V^{k}$ is the eigenspace of $V$ for $\sigma$ with eigenvalues $\eta^{k}$, where $\eta=\mbox{exp}(\frac{2\pi \sqrt{-1}}{t})$, $k=0,\ldots,t-1$.
	It is easy to see that the fixed points set $V^{0}:=V^{\sigma}=\{v\in V\ | \ \sigma(v)=v\}$
	is a vertex operator subalgebra of $V$
	
	Now we recall some notions regarding to twisted modules from \cite{L1}.
	\begin{definition}
		{\em Let $(V, {\bf 1}, Y)$ be a vertex algebra with an automorphism $\sigma$ of order $t$.
			A {\em $\sigma$-twisted $V$-module} is a triple $(W,d,Y_{W})$ consisting of a vector space $W$, an endomorphism 
			$d$ of $W$ and a linear map 
			$$Y_{W}(\cdot, z):V \longrightarrow (\mbox{End} W)[[z^{\frac{1}{t}}, z^{-\frac{1}{t}}]]$$
			satisfying the following conditions: 
			\begin{itemize}
				\item [(TW1)] 	$\mbox{For any } v\in V, w\in W$, $v_{n}w=0\;\mbox{for}\; n\in\frac{1}{t}\Z\;\mbox{sufficiently large};$
				\item [(TW2)] $Y_{W}(\textbf{1},z)=Id_{W};$ 
				\item [(TW3)] 	$[d, Y_{W}(v,z)]=Y_{W}(D(v),z)=\displaystyle\frac{d}{dz}Y_{W}(v,z)\;\mbox{for any}\; v\in V;$ 
				\item [(TW4)] For any $u,v\in V$, the following {\em$\sigma$-twisted Jacobi identity} holds:
				\begin{eqnarray*}
					&&z_{0}^{-1}\delta\left(\frac{z_{1}-z_{2}}{z_{0}}\right)Y_{W}(u,z_{1})Y_{W}(v,z_{2}) -
					z_{0}^{-1}\delta\left(\frac{z_{2}-z_{1}}{-z_{0}}\right)Y_{W}(v,z_{2})Y_{W}(u,z_{1})\\
					&&\hspace{2cm}= z_{2}^{-1}\frac{1}{t}\sum_{k=0}^{t-1}\delta\left(\left(\frac{z_{1}-z_{0}}{z_{2}}\right)^{\frac{1}{t}}\right)Y_{W}(Y(\sigma^{k}u,z_{0})v,z_{2}).
				\end{eqnarray*}
			\end{itemize}
		}
	\end{definition}
	
	If $V$ is a vertex operator algebra, a $\sigma$-twisted $V$-module for $V$ as a vertex algebra is called a 
	{\em $\sigma$-twisted weak module for $V$} as a vertex operator algebra.

	\begin{definition}
		{\em For $V$ a vertex operator algebra, a {\em $\sigma$-twisted $V$-module} $W$ is a $\sigma$-twisted weak module for $V$ as a vertex algebra
			and $W=\coprod\limits_{h\in \C}W_{(h)}$ such that
			\begin{itemize}
				\item [(TW5)]  $L(0)w=h w\;\;\;\;\mbox{for}\; h\in\C, w\in W_{(h)};$
				\item [(TW6)]  $\mbox{For any fixed } h, W_{(h+n)}=0\;\mbox{for} \;n\in \frac{1}{t}\Z\;\mbox{sufficiently small};$
				\item [(TW7)]  $\mbox{dim}\; W_{(h)}<\infty\;\;\mbox{for any} \;h\in\C.$
			\end{itemize}
		}
	\end{definition}

	\begin{rmk}
		\label{irrtwi}
		{\em
			Let $W$ be a $\sigma$-twisted $V$-module. Then 
			\begin{eqnarray}
			[L(0),Y_{W}(v,z)]=zY_{W}(L(-1)v,z)+Y_{W}(L(0)v,z)
			\end{eqnarray}
			for $v\in V$. Hence, if $v\in V$ is homogeneous, 
			\begin{eqnarray*}
				v_{n}W_{(h)}\subseteq W_{(h+ wt\;v-n-1)}\;\;\mbox{for}\; n\in\frac{1}{t}\Z,h\in\C.
			\end{eqnarray*}
			It follows that a $\sigma$-twisted $V$-module $W$ decomposes into twisted submodules corresponding to the congruence classes mod $\frac{1}{t}\Z$: For $h\in\mathbb{C}/\frac{1}{t}\Z$, let
			\begin{eqnarray}
			W_{[h]}=\coprod_{\alpha+\frac{1}{t}\Z=h}W_{(\alpha)}.
			\end{eqnarray}
			Then
			\begin{eqnarray}
			W=\coprod_{h\in\C/\frac{1}{t}\Z}W_{[h]}.
			\end{eqnarray}
			In particular, if $W$ is irreducible, then
			\begin{eqnarray}
			W=W_{[h]}
			\end{eqnarray}	
			for some $h$.
		}
	\end{rmk}
	
	\begin{rmk}
		{\em Let $W$ be a $\sigma$-twisted $V$-module.
			For $u\in V^{k}$, $v\in V$, there is the following {\em twisted iterate formula} (cf. (2.32) of \cite{L1}) 
			\begin{eqnarray}
			Y_{W}(Y(u,z_{0})v,z_{2})=\mbox{Res}_{z_{1}}\left(\frac{z_{1}-z_{0}}{z_{2}}\right)^{\frac{k}{t}}\cdot X
			\end{eqnarray}
			where
			\begin{eqnarray*}
				X=z_{0}^{-1}\delta\left(\frac{z_{1}-z_{2}}{z_{0}}\right)Y_{W}(u,z_{1})Y_{W}(v,z_{2}) -
				z_{0}^{-1}\delta\left(\frac{z_{2}-z_{1}}{-z_{0}}\right)Y_{W}(v,z_{2})Y_{W}(u,z_{1}).
			\end{eqnarray*}
			Then it can be easily deduced that for any $n\in\Z$,
			\begin{eqnarray}
			&&{}Y_{W}(u_{n}v,z)=
			\mbox{Res}_{z_{1}}\sum_{j=0}^{\infty}(-1)^{j}
			\binom{\frac{k}{t}}{j}z_{1}^{\frac{k}{t}-j}z^{-\frac{k}{t}}\cdot Y_{j}  \label{iter}
			\end{eqnarray}
			where
			\begin{eqnarray*}
				Y_{j}=\Big((z_{1}-z)^{n+j}Y_{W}(u,z_{1})Y_{W}(v,z)-
				(-z+z_{1})^{n+j}Y_{W}(v,z)Y_{W}(u,z_{1})\Big),  
			\end{eqnarray*}
			$\displaystyle\binom{\frac{k}{t}}{j}=\displaystyle\frac{\frac{k}{t}(\frac{k}{t}-1)\cdots(\frac{k}{t}-j+1)}{j!}$.
			Note that when $k=0$, $\displaystyle\frac{k}{t}=0$, we have $j=0$, and then (\ref{iter}) is the usual formula for (untwisted) $V$-modules (cf. (3.8.16) of \cite{LL}).
		}
	\end{rmk}
	
	\begin{definition}
		{\em 
			A {\em homomorphism} between two $\sigma$-twisted weak $V$-modules $M$ and $W$ is a linear map $f:M\longrightarrow W$ such that for any $v\in V$,
			\begin{eqnarray}
			f Y_{M}(v,z)=Y_{W}(v,z)f.
			\end{eqnarray}
			If $V$ is a vertex operator algebra, then a $V$-module homomorphism $f$ is compatible with the gradings:
			\begin{eqnarray}
			f(M_{(h)})\subseteq W_{(h)}\;\;\mbox{for}\;h\in\C.
			\end{eqnarray}
		}
	\end{definition}
	
	Using the formula (\ref{iter}), similarly as Proposition 4.5.1 of \cite{LL}, there is the following result.
	\begin{prop}
		\label{hom}
		Let $W_{1}$ and $W_{2}$ be $\sigma$-twisted $V$-modules and let $\psi\in Hom_{\C}(W_{1},W_{2})$. Suppose that
		$$Y(a,z)\psi=\psi Y(a,z),\;\;\mbox{for } a\in S,$$
		where $S$ is a given generating set of $V$. Then $\psi$ is a  $\sigma$-twisted $V$-module homomorphism.	
	\end{prop}

	In the following, we review from Section 3 of \cite{L1} the local systems of twisted vertex operators. 
	
	\begin{definition}
		{\em	Let $W$ be a vector space, let $t$ be a fixed positive integer. A {\em $\Z_t$-twisted weak vertex operator on $W$}
			is a formal series $a(z)=\sum_{n\in\frac{1}{t}\Z}a_{n}z^{-n-1}\in (\mbox{End}\; W)[[z^{\frac{1}{t}}, z^{-\frac{1}{t}}]]$ such that for any $w\in W$, $a_{n}w=0$ for $n\in\frac{1}{t}\Z$ sufficiently large.
		}
	\end{definition}
	
	\begin{definition}
		{\em Two $\Z_t$-twisted weak vertex operators $a(z)$ and $b(z)$ are said to be {\em mutually local} if there is a positive integer $n$
			such that
			$$(z_{1}-z_{2})^{n}a(z_{1})b(z_{2})=(z_{1}-z_{2})^{n}b(z_{2})a(z_{1}).$$
			A $\Z_t$-twisted weak vertex operator is called a {\em $\Z_t$-twisted vertex operator} if it is local with itself.
			
		}
	\end{definition}
	
	Denote by $F(W,t)$ the space of all $\Z_{t}$-twisted weak vertex operators on $W$. 
	Let $\sigma$ be the endomorphism of $(\mbox{End}\; W)[[z^{\frac{1}{t}},z^{-\frac{1}{t}}]]$ defined by:
	$\sigma f(z^{\frac{1}{t}})=f(\eta^{-1}z^{\frac{1}{t}})$.
	Denote by $F(W,t)^{k}=\{f(z)\in F(W,t)\ | \ \sigma f(z)=\eta^{k}f(z)\}$ for $0\leq k\leq t-1$.
	For any mutually local $\Z_t$-twisted vertex operators $a(z), b(z)$ on $W$, define $a(z)_{n}b(z)$
	as follows (cf. Definition 3.7 of \cite{L1}).
	
	\begin{definition}
		{\em
			Let $W$ be a vector space and let $a(z)$ and $b(z)$ be mutually local $\Z_t$-twisted vertex operators on $W$ such that $a(z)\in F(W,t)^{k}$. Then for any integer $n$ we define $a(z)_{n}b(z)$ as an element of $F(W,t)$ as follows:
			\begin{eqnarray}
			a(z)_{n}b(z)=Res_{z_{1}}Res_{z_{0}}\left(\frac{z_{1}-z_{0}}{z}\right)^{\frac{k}{t}}z_{0}^{n}\cdot X
			\end{eqnarray}
			where
			$$X=z_{0}^{-1}\delta\left(\frac{z_{1}-z}{z_{0}}\right)a(z_{1})b(z)-z_{0}^{-1}\delta\left(\frac{z-z_{1}}{-z_{0}}\right)b(z)a(z_{1}).$$
		}
	\end{definition}
	
	For any set $S$ of mutually local $\Z_{t}$-twisted vertex operators on $W$, by Zorn's Lemma, there exists a local system $A$ of $\Z_{t}$-twisted vertex operators on $W$ (cf. Section 3 of \cite{L1}). Denote by
	$\langle S\rangle$ the vertex algebra generated by $S$ inside $A$ via the operations $a(z)_{n}b(z)$, $n\in\Z$.
	Furthermore, there is the following result (cf. Corollary 3.15 of \cite{L1}).
	\begin{prop}
		\label{gen}
		Let $W$ be any vector space. Let $S$ be a set of mutually local	$\Z_{t}$-twisted vertex operators on $W$. Then $\langle S\rangle$ is a vertex algebra with an automorphism $\sigma$ of order $t$ such that $W$ is a $\sigma$-twisted $\langle S\rangle$-module in the sense $Y_{W}(a(z),z_{1})=a(z_{1})$.
	\end{prop}
	
	\section{Automorphism group}
	\def\theequation{3.\arabic{equation}}
	\setcounter{equation}{0}
	
	In this section, we first recall the definition of the twisted Heisenberg-Virasoro algebra $\mathcal{L}$ and the construction of the twisted Heisenberg-Virasoro vertex operator algebra $V_{\mathcal{L}}(\ell_{123},0)$. Then we determine the automorphism group of the vertex operator algebra $V_{\mathcal{L}}(\ell_{123},0)$.
	
	Recall the definition of
	the twisted Heisenberg-Virasoro algebra $\mathcal{L}$ (see \cite{ACKP} or \cite{B1}).
	
	\begin{definition}\label{Twisted Heisenberg-Virasoro algebra}
		{\em The twisted Heisenberg-Virasoro algebra $\mathcal{L}$ is a Lie algebra with basis
			$\{L_{n},I_{n},c_{1},c_{2},c_{3}\ | \ n\in\Z  \}$, and the following Lie brackets:
			\begin{eqnarray}
			[L_{m}, L_{n}]=(m-n)L_{m+n}+\delta_{m+n,0}\frac{m^{3}-m}{12}c_{1},  \label{eq:2.4}
			\end{eqnarray}
			\begin{eqnarray}
			[L_{m}, I_{n}]=-n I_{m+n}-\delta_{m+n,0}(m^{2}+m)c_{2},   \label{eq:2.5}
			\end{eqnarray}
			\begin{eqnarray}
			[I_{m}, I_{n}]=m\delta_{m+n,0}c_{3},\;\;\; [\mathcal{L}, c_{i}]=0,\; i=1,2,3.   \label{eq:2.6}
			\end{eqnarray}
		}
	\end{definition}
	
	Clearly, $\mbox{Span}\{L_{n},\;c_{1}\ | \ n\in\Z\}$ is a Virasoro algebra, $\mbox{Span}\{I_{n},\;c_{3}\ | \ n\in\Z \backslash\{0\}\}$ is
	an infinite-dimensional Heisenberg algebra, we denote them by $Vir$, $\mathcal{H}$ respectively.
	
	Let
	$$L(z)=\sum\limits_{n\in\Z}L_{n}z^{-n-2},\;\; I(z)=\sum\limits_{n\in\Z}I_{n}z^{-n-1},$$
	then the defining relations of $\mathcal{L}$ become to be
	\begin{eqnarray}
	&&{}[L(z_{1}), L(z_{2})]  \nonumber \\
	&&{}=\sum\limits_{m,n\in\Z}(m-n)L_{m+n}z_{1}^{-m-2}z_{2}^{-n-2}+\sum\limits_{m\in\Z}\frac{m^{3}-m}{12}c_{1}z_{1}^{-m-2}z_{2}^{m-2} \nonumber \\
	&&{}=\displaystyle\frac{d}{dz_{2}}\Big(L(z_{2})\Big)z_{1}^{-1}\delta\left(\frac{z_{2}}{z_{1}}\right)
	+ 2L(z_{2})\frac{\partial}{\partial z_{2}}z_{1}^{-1}\delta\left(\frac{z_{2}}{z_{1}}\right)
	\nonumber \\
	&&{}\;\;\;\;
	+\frac{c_{1}}{12}\left(\frac{\partial}{\partial z_{2}}\right)^{3}z_{1}^{-1}\delta\left(\frac{z_{2}}{z_{1}}\right),    \label{eq:2.7}
	\end{eqnarray}
	
	\begin{eqnarray}
	&&{}[L(z_{1}), I(z_{2})]  \nonumber \\
	&&{}=-\sum\limits_{m,n\in\Z}n I_{m+n}z_{1}^{-m-2}z_{2}^{-n-1}-\sum\limits_{m\in\Z}(m^{2}+m)c_{2}z_{1}^{-m-2}z_{2}^{m-1}  \nonumber \\
	&&{}=\displaystyle\frac{d}{dz_{2}}\Big(I(z_{2})\Big)z_{1}^{-1}\delta\left(\frac{z_{2}}{z_{1}}\right)
	+ I(z_{2})\frac{\partial}{\partial z_{2}}z_{1}^{-1}\delta\left(\frac{z_{2}}{z_{1}}\right)
	\nonumber\\
	&&{}\;\;\;\;
	-\left(\frac{\partial}{\partial z_{2}}\right)^{2}z_{1}^{-1}\delta\left(\frac{z_{2}}{z_{1}}\right)c_{2},    \label{eq:2.8}
	\end{eqnarray}
	\begin{eqnarray}
	[I(z_{1}), I(z_{2})]
	=\sum\limits_{m\in\Z}m c_{3}z_{1}^{-m-1}z_{2}^{m-1}
	=\frac{\partial}{\partial z_{2}}z_{1}^{-1}\delta\left(\frac{z_{2}}{z_{1}}\right)c_{3}.  \label{eq:2.9}
	\end{eqnarray}

	We recall the construction of the twisted Heisenberg-Virasoro vertex operator algebra $V_{\mathcal{L}}(\ell_{123},0)$ from \cite{GW}.
	Let
	\begin{eqnarray*}
		\mathcal{L}_{(\leq 1)}=\coprod_{n\leq 1}\C L_{-n}\oplus\coprod_{n\leq 0}\C I_{-n}\oplus\sum_{i=1}^{3}\C c_{i},  \label{eq:2.10}
	\end{eqnarray*}
	\begin{eqnarray*}
		\mathcal{L}_{(\geq 2)}=\coprod_{n\geq 2}\C L_{-n}\oplus\coprod_{n\geq 1}\C I_{-n}.       \label{eq:2.11}
	\end{eqnarray*}
	They are graded subalgebras of $\mathcal{L}$
	and 
	\begin{eqnarray*}
		\mathcal{L}=\mathcal{L}_{(\leq 1)}\oplus \mathcal{L}_{(\geq 2)}.                  \label{eq:2.12}
	\end{eqnarray*}

	Let $\ell_{i},i=1,2,3,$ be any complex numbers. Consider $\C$ as an $\mathcal{L}_{(\leq 1)}$-module with $c_{i}$ acting
	as the scalar $\ell_{i},i=1,2,3,$ and with $\coprod_{n\leq 1}\C L_{-n}\oplus\coprod_{n\leq 0}\C I_{-n}$
	acting trivially. Denote this $\mathcal{L}_{(\leq 1)}$-module by $\C_{\ell_{123}}$.
	Form the induced module
	\begin{eqnarray}
	V_{\mathcal{L}}(\ell_{123},0)=U(\mathcal{L})\otimes_{U(\mathcal{L}_{(\leq 1)})}\C_{\ell_{123}},   \label{eq:2.13}
	\end{eqnarray}
	where $U(\cdot)$ denotes the universal enveloping algebra of a Lie algebra. Set ${\bf 1} =1\otimes 1\in V_{\mathcal{L}}(\ell_{123},0)$.
	$V_{\mathcal{L}}(\ell_{123},0)$ is a vertex operator algebra with vacuum vector ${\bf 1}$ 
	and conformal vector $\omega=L_{-2}{\bf 1}$.
	And $\{\omega=L_{-2}{\bf 1},I:=I_{-1}{\bf 1}\}$ is a generating subset of
	$V_{\mathcal{L}}(\ell_{123},0)$.
	Recall the grading on $V_{\mathcal{L}}(\ell_{123},0)$:
	\begin{eqnarray*}
		V_{\mathcal{L}}(\ell_{123},0)=\coprod_{n\geq 0}V_{\mathcal{L}}(\ell_{123},0)_{(n)},
	\end{eqnarray*}
	where $V_{\mathcal{L}}(\ell_{123},0)_{(0)}=\C$ and $V_{\mathcal{L}}(\ell_{123},0)_{(n)}$, $n\geq 1$, has a basis consisting of the vectors
	$$I_{-k_{1}}\cdots I_{-k_{s}}L_{-m_{1}}\cdots L_{-m_{r}}{\bf{1}} $$
	for
	$r, s\geq 0$,
	$m_{1}\geq\cdots\geq m_{r}\geq 2$, $k_{1}\geq\cdots\geq k_{s}\geq 1$ with $\sum\limits_{i=1}^{r}m_{i}+\sum\limits_{j=1}^{s}k_{j}=n.$
	
	Now we give our first main result. The automorphism group of $V_{\mathcal{L}}(\ell_{123},0)$ is determined in the following theorem.
	\begin{thm} 
		\label{autogp}
		\begin{itemize}
			\item [(1)] If $\ell_{2}\neq 0$, then $Aut(V_{\mathcal{L}}(\ell_{123},0))=\{id\}.$
			\item [(2)] If $\ell_{2}= 0$ and $\ell_{3}\neq 0$, then 
			$Aut(V_{\mathcal{L}}(\ell_{123},0))\cong \mathbb{Z}_{2}$.
			\item [(3)] If both $\ell_{2}$ and $\ell_{3}$ are 0, then $Aut(V_{\mathcal{L}}(\ell_{123},0))\cong\mathbb{C}^{\times}=\mathbb{C}\backslash \{0\}.$
		\end{itemize}
	\end{thm}
	\begin{proof}
		Let $$\varphi: V_{\mathcal{L}}(\ell_{123},0)\longrightarrow V_{\mathcal{L}}(\ell_{123},0)$$
		be an automorphism of the vertex operator algebra $V_{\mathcal{L}}(\ell_{123},0)$.
		Then $\varphi({\bf 1})={\bf 1}$ and $\varphi(\omega)=\omega.$
		Since $V_{\mathcal{L}}(\ell_{123},0)$ is generated by $\omega$ and $I=I_{-1}{\bf 1}$, it suffices to determine
		$\varphi(I)$.
		$\varphi$ is grading-preserving, so 
		$\varphi(I)=aI$ for some $a\in\mathbb{C}^{\times}$.
		
		Then, on the one hand, we have
		$$\varphi(L_{1}I)=aL_{1}I=a[L_{1},I_{-1}]{\bf 1}=-2a\ell_{2}{\bf 1},$$
		on the other hand, we have
		$$\varphi(L_{1}I)=\varphi([L_{1},I_{-1}]{\bf 1})=-2\ell_{2}{\bf 1}.$$
		Therefore if $\ell_{2}\neq 0$, we get that $a=1$,
		i.e. when $\ell_{2}\neq 0$, $\mbox{Aut}(V_{\mathcal{L}}(\ell_{123},0))=\{id\}$ only consists of the identity map.
		
		Suppose now $\ell_{2}= 0$. Let's consider $\varphi(I_{1}I)$.
		On the one hand,
		$$\varphi(I_{1}I)=\varphi(I)_{1}\varphi(I)=a^{2}I_{1}I_{-1}{\bf 1}=a^{2}[I_{1}, I_{-1}]{\bf 1}=a^{2}\ell_{3}{\bf 1},$$
		on the other hand,
		$$\varphi(I_{1}I)=\varphi([I_{1}, I_{-1}]{\bf 1})=\ell_{3}{\bf 1}.$$
		So if $\ell_{3}\neq 0$, we get $a^{2}=1$, i.e. $a=1$
		or $a=-1$, then $\mbox{Aut}(V_{\mathcal{L}}(\ell_{123},0))\cong\mathbb{Z}_{2}$.

		Now	let $\ell_{2}=0 \;\mbox{and}\; \ell_{3}=0$, then $a$ can be any nonzero complex number, so $\mbox{Aut}(V_{\mathcal{L}}(\ell_{123},0))\cong\mathbb{C}^{\times}.$
		
	\end{proof}

	\section{$\sigma_{t}$-twisted $V_{\mathcal{L}}(\ell_{123},0)$-modules}
	\def\theequation{4.\arabic{equation}}
	\setcounter{equation}{0}
	
	In this section, we always require $\ell_{2}=0$.
	We study twisted modules for the vertex operator algebra $V_{\mathcal{L}}(\ell_{123},0)$.
	More precisely, for any integer $t>1$, we introduce an infinite-dimensional Lie algebra $\mathcal{L}_{t}$. We show that there is a one-to-one correspondence between restricted $\mathcal{L}_{t}$-modules of level $\ell_{13}$
	and $\sigma_{t}$-twisted $V_{\mathcal{L}}(\ell_{123},0)$-modules, where
	$\sigma_{t}$ is an order $t$ automorphism of $V_{\mathcal{L}}(\ell_{123},0)$. And we give a complete list of irreducible $\sigma_{t}$-twisted $V_{\mathcal{L}}(\ell_{123},0)$-modules.
	
	Note that if $\ell_{2}=0$ and $\ell_{3}\neq0$, then $t$ can only be the integer 2 (Theorem \ref{autogp}).
	As we need, we introduce the following Lie algebra.
	\begin{definition}
		{\em	Let $\mathcal{L}_{t}$ be a Lie algebra with basis
			$\{L_{n},I_{n+\frac{1}{t}},k_{1},k_{3}\ | \ n\in\Z  \}$, and the Lie brackets are given by:
			\begin{eqnarray}
			[L_{m}, L_{n}]=(m-n)L_{m+n}+\delta_{m+n,0}\frac{m^{3}-m}{12}k_{1},  \label{Lt1}
			\end{eqnarray}
			\begin{eqnarray}
			[L_{m}, I_{n+\frac{1}{t}}]=-(n+\frac{1}{t}) I_{m+n+\frac{1}{t}},
			\label{Lt2}   
			\end{eqnarray}
			\begin{eqnarray}
			[I_{m+\frac{1}{t}}, I_{n+\frac{1}{t}}]=(m+\frac{1}{t})\delta_{m+n+\frac{2}{t},0}\delta_{t,2}k_{3},\;\;\; [\mathcal{L}, k_{i}]=0,\; i=1,3.  \label{Lt3}
			\end{eqnarray}
		}
	\end{definition}
	
	Note that if $t\neq 2$, then $[I_{m+\frac{1}{t}}, I_{n+\frac{1}{t}}]=0$ for any $m,n\in\Z$.  
	
	Form the generating function as
	$$L(z)=\sum\limits_{n\in\Z}L_{n}z^{-n-2},\;\;I_{\sigma_{t}}(z)=\sum\limits_{n\in\Z}I_{n+\frac{1}{t}}z^{-n-\frac{1}{t}-1}.$$ 
	Then the defining relations (\ref{Lt1}), (\ref{Lt2}), (\ref{Lt3}) amount to:
	\begin{eqnarray}
	&&{}[L(z_{1}), L(z_{2})]  \nonumber \\
	&&{}=\frac{d}{dz_{2}}(L(z_{2}))z_{1}^{-1}\delta\left(\frac{z_{2}}{z_{1}}\right)
	+ 2L(z_{2})\frac{\partial}{\partial z_{2}}z_{1}^{-1}\delta\left(\frac{z_{2}}{z_{1}}\right)
	\nonumber \\
	&&{}
	\;\;\;\;+\frac{1}{12}\left(\frac{\partial}{\partial z_{2}}\right)^{3}z_{1}^{-1}\delta\left(\frac{z_{2}}{z_{1}}\right)k_{1},   \label{eq:4.4}
	\end{eqnarray}
	\begin{eqnarray}
	&&{}[L(z_{1}), I_{\sigma_{t}}(z_{2})]  \nonumber \\
	&&{}=-\sum\limits_{m,n\in\Z}(n+\frac{1}{t}) I_{m+n+\frac{1}{t}}z_{1}^{-m-2}z_{2}^{-n-\frac{1}{t}-1}\nonumber \\
	&&{}=\frac{d}{dz_{2}}\Big(I_{\sigma_{t}}(z_{2})\Big)z_{1}^{-1}\delta\left(\frac{z_{2}}{z_{1}}\right)
	+ I_{\sigma_{t}}(z_{2})\frac{\partial}{\partial z_{2}}z_{1}^{-1}\delta\left(\frac{z_{2}}{z_{1}}\right),  \label{eq:4.5}
	\end{eqnarray}
	\begin{eqnarray}
	&&{}[I_{\sigma_{t}}(z_{1}), I_{\sigma_{t}}(z_{2})]
	=\sum\limits_{m\in\Z}(m+\frac{1}{t}) z_{1}^{-m-\frac{1}{t}-1}z_{2}^{m+\frac{1}{t}-1}\delta_{t,2}k_{3}
	\nonumber\\
	&&{}
	=\frac{\partial}{\partial z_{2}}\left(z_{1}^{-1}\delta\left(\frac{z_{2}}{z_{1}}\right)\left(\frac{z_{2}}{z_{1}}\right)^{\frac{1}{t}}\right)\delta_{t,2}k_{3}. \label{eq:4.6} 
	\end{eqnarray}
	
	Now we construct irreducible $\mathcal{L}_{t}$-modules (cf. \cite{GW}, \cite{LL}, etc.).
	Let
	$$(\mathcal{L}_{t})_{\geq 0}=(\coprod_{m\geq 0}\C L_{m})\oplus(\coprod_{n\geq 0}\C I_{n+\frac{1}{t}})\oplus \C k_{1}\oplus \C k_{3}.$$
	It is a subalgebra of $\mathcal{L}_{t}.$

	Let $\C$ be an $(\mathcal{L}_{t})_{\geq 0}$-module, where $L_{m}$, $I_{n+\frac{1}{t}}$ act trivially for all $m\geq 1$, $n\geq 0$,
	and $L_{0}, k_{1}, k_{3}$ act as scalar multiplications by $h, \Bbbk_{1}, \Bbbk_{3}$ respectively. Denote this $(\mathcal{L}_{t})_{\geq 0}$-module by $\mathbb{C}_{\Bbbk_{13},h}$.
	Form the induced module
	$$M_{\mathcal{L}_{t}}(\Bbbk_{1}, \Bbbk_{3}, h)=U(\mathcal{L}_{t})\otimes_{U((\mathcal{L}_{t})_{\geq 0})}\C_{\Bbbk_{13},h}.$$
	Set 
	$${\bf 1}_{\Bbbk_{13},h}=1\in\mathbb{C}_{\Bbbk_{13},h}\subset M_{\mathcal{L}_{t}}(\Bbbk_{1}, \Bbbk_{3}, h).$$
	Then $M_{\mathcal{L}_{t}}(\Bbbk_{1}, \Bbbk_{3}, h)$ is $\C$-graded by $L_{0}$-eigenvalues:
	$$M_{\mathcal{L}_{t}}(\Bbbk_{1}, \Bbbk_{3}, h)=\coprod_{n\geq 0}M_{\mathcal{L}_{t}}(\Bbbk_{1}, \Bbbk_{3}, h)_{n+h},$$
	where $M_{\mathcal{L}_{t}}(\Bbbk_{1}, \Bbbk_{3}, h)_{(h)}=\mathbb{C}_{\Bbbk_{13},h}$ and $M_{\mathcal{L}_{t}}(\Bbbk_{1}, \Bbbk_{3}, h)_{(n+h)}$ is the $L_{0}$-eigenspace of eigenvalue $n+h$ for $n> 0$.
	$M_{\mathcal{L}[-1]}(\Bbbk_{1}, \Bbbk_{3}, h)_{(n+h)}$ has a basis consisting of
	$$I_{-k_{1}+\frac{1}{t}}\cdots I_{-k_{s}+\frac{1}{t}}L_{-m_{1}}\cdots L_{-m_{r}}{\bf{1}}_{\Bbbk_{13},h} $$
	for
	$r, s\geq 0$,
	$m_{1}\geq\cdots\geq m_{r}\geq 1$, $k_{1}\geq\cdots\geq k_{s}\geq 1$ with $\sum\limits_{i=1}^{r}m_{i}+\sum\limits_{j=1}^{s}(k_{j}-\displaystyle\frac{1}{t})=n$, $n> 0$.
	
	\begin{rmk}
		\label{universal}
		{\em
			As a module for $\mathcal{L}_{t}$, $M_{\mathcal{L}_{t}}(\Bbbk_{1}, \Bbbk_{3}, h)$ is generated by ${\bf 1}_{\Bbbk_{13},h}$ with the relations
			\begin{equation*}
			L_{0}{\bf 1}_{\Bbbk_{13},h}=h{\bf 1}_{\Bbbk_{13},h},\;k_{i}=\Bbbk_{i},\;i=1,3,
			\end{equation*}
			and
			\begin{equation*}
			L_{m}{\bf 1}_{\Bbbk_{13},h}=0, I_{n+\frac{1}{t}}{\bf 1}_{\Bbbk_{13},h}=0\; \mbox{for}\;\; m\geq 1, n\geq 0.
			\end{equation*}
			$M_{\mathcal{L}	_{t}}(\Bbbk_{1}, \Bbbk_{3}, h)$ is {\em universal}
			in the sense that for any $\mathcal{L}_{t}$-module $W$ of level $\Bbbk_{13}$ equipped with a vector $v$ such that 
			$L_{0}v=hv,\;
			L_{m}v=0, I_{n+\frac{1}{t}}v=0\; \mbox{for}\; m\geq 1, n\geq 0$,
			there exists a unique $\mathcal{L}_{t}$-module map $M_{\mathcal{L}_{t}}(\Bbbk_{1}, \Bbbk_{3}, h)\longrightarrow W$ sending ${\bf 1}_{\Bbbk_{13},h}$ to $v$.
		}
	\end{rmk}
	
	In general, $M_{\mathcal{L}_{t}}(\Bbbk_{1}, \Bbbk_{3}, h)$ as an $\mathcal{L}_{t}$-module may be reducible.
	Since $\mathbb{C}_{\Bbbk_{13},h}$ generate $M_{\mathcal{L}_{t}}(\Bbbk_{1}, \Bbbk_{3}, h)$ as 
	$\mathcal{L}_{t}$-module, for any proper submodule $U$ of $M_{\mathcal{L}_{t}}(\Bbbk_{1}, \Bbbk_{3}, h)$, 
	$U_{(h)}=U\bigcap M_{\mathcal{L}_{t}}(\Bbbk_{1}, \Bbbk_{3}, h)_{(h)}=0$.
	Hence there exists a maximal proper $\mathcal{L}_{t}$-submodule $T_{\mathcal{L}_{t}}(\Bbbk_{1}, \Bbbk_{3}, h)$.
	Set
	$$L_{\mathcal{L}_{t}}(\Bbbk_{1}, \Bbbk_{3}, h)=M_{\mathcal{L}_{t}}(\Bbbk_{1}, \Bbbk_{3}, h)/T_{\mathcal{L}_{t}}(\Bbbk_{1}, \Bbbk_{3}, h).$$
	Then $L_{\mathcal{L}_{t}}(\Bbbk_{1}, \Bbbk_{3}, h)$ is an irreducible $\mathcal{L}_{t}$-module.
	
	\begin{definition}
		{\em An $\mathcal{L}_{t}$-module $W$ is said to be {\em restricted} if for any $w\in W, n\in\Z$,
			$L_{n}w=0$ and $I_{n+\frac{1}{t}}w=0$ for $n$ sufficiently large.
			We say an $\mathcal{L}_{t}$-module $W$ is of {\em level} $\Bbbk_{13}$ if the central element $k_{i}$ acts as
			scalar $\Bbbk_{i}$ for $i=1,3.$}
	\end{definition}
	
	It is easy to see that $M_{\mathcal{L}_{t}}(\Bbbk_{1}, \Bbbk_{3}, h)$, $L_{\mathcal{L}_{t}}(\Bbbk_{1}, \Bbbk_{3}, h)$ are restricted $\mathcal{L}_{t}$-modules of level $\Bbbk_{13}$, for any $h\in\C$. 
	Now we are going to relate $\mathcal{L}_{t}$-modules with twisted $V_{\mathcal{L}}(\ell_{123},0)$-modules.
	On the one hand, we have 
	\begin{thm}
		\label{1}
		If $W$ is a restricted $\mathcal{L}_{t}$-module of level $\ell_{13}$, then $W$ is a $\sigma_{t}$-twisted
		$V_{\mathcal{L}}(\ell_{123},0)$-module for $V_{\mathcal{L}}(\ell_{123},0)$ as a vertex algebra with
		$$Y_{\sigma_{t}}(L_{-2}{\bf 1}, z)=L(z)=\sum\limits_{n\in\Z}L_{n}z^{-n-2},$$
		$$Y_{\sigma_{t}}(I_{-1}{\bf 1}, z)=I_{\sigma_{t}}(z)=\sum\limits_{n\in\Z}I_{n+\frac{1}{t}}z^{-n-\frac{1}{t}-1}.$$
	\end{thm}
	\begin{proof}
		Let $U_{W}=\{L(z), I_{\sigma_{t}}(z), {\bf 1}_{W}\}$, where ${\bf 1}_{W}$ is the identity operator on $W$. 
		Clearly, $L(z), I_{\sigma_{t}}(z)$ are $\Z_{t}$-twisted weak vertex operators on $W$.
		From (\ref{eq:4.4}), (\ref{eq:4.5}), (\ref{eq:4.6}), using (\ref{eq:2.1}), we see that
		$L(z), I_{\sigma_{t}}(z)$ are mutually local $\Z_{t}$-twisted vertex operators on $W$.
		Hence, by Proposition \ref{gen}, $\langle U_{W}\rangle$ is a vertex algebra with $W$ a faithful $\sigma$-twisted module, where $\sigma$ is an order $t$ automorphism of the vertex algebra $\langle U_{W}\rangle$.
		To say that $W$ is a $\sigma_{t}$-twisted module for $V_{\mathcal{L}}(\ell_{123},0)$, from Proposition 3.17 of \cite{L1},
		it suffices to show that there exists a vertex algebra homomorphism from $V_{\mathcal{L}}(\ell_{123},0)$
		to $\langle U_{W}\rangle$.

		By Lemma 2.11 of \cite{L1},
		$Y(L(z),z_{1})$ and $Y(I_{\sigma_{t}}(z),z_{1})$ satisfy the twisted Heisenberg-Virasoro relations (\ref{eq:2.7}), (\ref{eq:2.8}), (\ref{eq:2.9}).
		Then $\langle U_{W}\rangle$ is an $\mathcal{L}$-module with $L_{n}, I_{n}$ acting as $L(z)_{n+1}$, $I_{\sigma_{t}}(z)_{n}$ for 
		$n\in\Z,$ $c_{i}$ acting as $\ell_{i}$ with $\ell_{2}=0$, $i=1,2,3.$

		By the universal property of $V_{\mathcal{L}}(\ell_{123},0)$ (c.f. Remark 2.7 of \cite{GW}), there exists a unique $\mathcal{L}$-module
		homomorphism 
		$$\psi: V_{\mathcal{L}}(\ell_{123},0)\longrightarrow \langle U_{W}\rangle;\;\;{\bf 1}\mapsto {\bf 1}_{W}.$$
		Then
		$$\psi(\omega_{n}v)=L(z)_{n}\psi(v)=\psi(\omega)_{n}\psi(v),$$
		$$\psi(I_{n}v)=I_{\sigma_{t}}(z)_{n}\psi(v)=\psi(I)_{n}\psi(v).$$
		for all $v\in V_{\mathcal{L}}(\ell_{123},0)$, $n\in\Z.$
		Hence $\psi$ is a vertex algebra homomorphism.
		Therefore, $W$ is a weak $\sigma_{t}$-twisted
		$V_{\mathcal{L}}(\ell_{123},0)$-module with
		$Y_{\sigma_{t}}(L_{-2}{\bf 1}, z)=L(z)$,
		$Y_{\sigma_{t}}(I_{-1}{\bf 1}, z)=I_{\sigma_{t}}(z).$
	\end{proof}

	Conversely, we have
	\begin{thm}
		\label{2}
		If $W$ is a $\sigma_{t}$-twisted $V_{\mathcal{L}}(\ell_{123},0)$ ($\ell_{2}=0$)-module,
		then $W$ is a restricted $\mathcal{L}_{t}$-module of level $\ell_{13}$ 
		with $L(z)=Y_{W}(L_{-2}{\bf 1},z)$, $I_{\sigma_{t}}(z)=Y_{W}(I_{-1}{\bf 1},z).$
	\end{thm}
	\begin{proof}
		Let $W$ be a $\sigma_{t}$-twisted $V_{\mathcal{L}}(\ell_{123},0)(\ell_{2}=0)$-module.
		Recall the following formula (c.f. (2.40) of \cite{L1})
		$$[Y_{W}(a,z_{1}), Y_{W}(b,z_{2})]=\sum_{j=0}^{\infty}\frac{1}{j!}\left(\left(\frac{\partial}{\partial z_{2}}\right)^{j}z_{1}^{-1}
		\delta\left(\frac{z_{2}}{z_{1}}\right)\left(\frac{z_{2}}{z_{1}}\right)^{\frac{k}{t}}\right)
		Y_{W}(a_{j}b,z_{2}),$$
		where $k$ is determined by $a$.
		In our case, when $a=L_{-2}{\bf 1}$, $k=0$, when $a=I_{-1}{\bf 1}$, $k=1.$
		For $a=b=L_{-2}{\bf 1}$, we have ((2.21) of \cite{GW})
		$$ (L_{-2}{\bf 1})_{j}L_{-2}{\bf 1}= (j+1)L_{j-3}
		{\bf 1}+\delta_{j-3,0}\frac{(j-1)^{3}-(j-1)}{12}c_{1}{\bf 1},     $$
		so 
		\begin{eqnarray}
		&&{} [Y_{W}(L_{-2}{\bf 1}, z_{1}), Y_{W}(L_{-2}{\bf 1}, z_{2})]    \nonumber\\
		&&{} = Y_{W}(L_{-3}{\bf 1}, z_{2})z_{2}^{-1}\delta\left(\frac{z_{1}}{z_{2}}\right)  \nonumber\\
		&&{}\;\quad +2Y_{W}(L_{-2}{\bf 1}, z_{2})\left(\frac{\partial}{\partial z_{1}}\right)z_{2}^{-1}\delta\left(\frac{z_{1}}{z_{2}}\right)  \nonumber\\
		&&{}\;\quad +\frac{1}{12}\left(\frac{\partial}{\partial z_{1}}\right)^{3}z_{2}^{-1}\delta\left(\frac{z_{1}}{z_{2}}\right)\ell_{1}{\bf 1}.
		\end{eqnarray}
		For $a=L_{-2}{\bf 1}, b=I_{-1}{\bf 1}$, we have ((2.22) of \cite{GW} with $\ell_{2}=0$)
		$$ (L_{-2}{\bf 1})_{j}I_{-1}{\bf 1}= I_{j-2}{\bf 1},$$
		so
		\begin{eqnarray}
		&&{}[Y_{W}(L_{-2}{\bf 1}, z_{1}), Y_{W}(I_{-1}{\bf 1}, z_{2})]    \nonumber\\
		&&{}= Y_{W}(I_{-2}{\bf 1}, z_{2})z_{1}^{-1}\delta\left(\frac{z_{2}}{z_{1}}\right)  
		+Y_{W}(I_{-1}{\bf 1}, z_{2})\left(\frac{\partial}{\partial z_{2}}\right)z_{1}^{-1}\delta\left(\frac{z_{2}}{z_{1}}\right),
		\end{eqnarray} 
		
		For $a=b=I_{-1}{\bf 1}$, we have ((2.23) of \cite{GW})
		$$(I_{-1}{\bf 1})_{j}I_{-1}{\bf 1}= j\delta_{j-1,0}c_{3}{\bf 1},$$
		so 
		\begin{eqnarray}
		[Y_{W}(I_{-1}{\bf 1}, z_{1}), Y_{W}(I_{-1}{\bf 1}, z_{2})]  
		= \left(\frac{\partial}{\partial z_{2}}\right)z_{1}^{-1}\delta\left(\frac{z_{2}}{z_{1}}\right)\left(\frac{z_{2}}{z_{1}}\right)^{\frac{1}{t}}\ell_{3}{\bf 1}.
		\end{eqnarray} 
		
		Note that for $t\neq 2$, we have to require $\ell_{3}=0$ (Theorem \ref{autogp}).
		Therefore, with (\ref{eq:4.4}), (\ref{eq:4.5}) and (\ref{eq:4.6}), $W$ is a $\mathcal{L}_{t}$-module with $L(z)=Y_{W}(L_{-2}{\bf 1},z)$, $I_{\sigma_{t}}(z)=Y_{W}(I_{-1}{\bf 1},z)$, $k_{i}=\ell_{i}$, $i=1,3$.
		Then $W$ is restricted of level $\ell_{13}$ is clear.
	\end{proof}

	Let
	\begin{equation*}
	\mathcal{L}_{t}^{(0)}=\C L_{0}\oplus\C k_{1}\oplus\C k_{3}, \;\;\;\;
	\mathcal{L}_{t}^{(n)}=\C L_{-n}\;\;\mbox{for}\;0\neq n\in\Z,
	\end{equation*}
	\begin{equation*}
	\mathcal{L}_{t}^{(-\frac{1}{t}+n)}=\C I_{-n+\frac{1}{t}},\;\;\;\;
	\mathcal{L}_{t}^{(k)}=0\;\;\mbox{for}\; k\in\frac{1}{t}\Z, \;k\notin\Z, -\frac{1}{t}+\Z.
	\end{equation*}
	Then $\mathcal{L}=\coprod\limits_{n\in\Z}\mathcal{L}_{t}^{(\frac{n}{t})}$ is a $\displaystyle\frac{1}{t}\Z$-graded Lie algebra, and the grading is given by $\mbox{ad} L_{0}$-eigenvalues.

	By Theorem \ref{1} and Theorem \ref{2} we have the following result.
	\begin{thm}
		\label{voamod}
		The $\sigma_{t}$-twisted modules for $V_{\mathcal{L}}(\ell_{123},0)$ ($\ell_{2}=0$) viewed as a vertex operator algebra
		(i.e. $\C$-graded by $L_{0}$-eigenvalues and with the two grading restrictions (TW6), (TW7)) are exactly those restricted modules
		for the Lie algebra $\mathcal{L}_{t}$ of level $\ell_{13}$ that are $\C$-graded by $L_{0}$-eigenvalues and with the two grading restrictions.
		Furthermore, for any $\sigma_{t}$-twisted $V_{\mathcal{L}}(\ell_{123},0)$-module $W$, the $\sigma_{t}$-twisted $V_{\mathcal{L}}(\ell_{123},0)$-submodules of $W$ are exactly
		the submodules of $W$ for $\mathcal{L}_{t}$, and these submodules are in particular graded.
	\end{thm}

	Hence irreducible restricted $\mathcal{L}_{t}$-modules of level $\ell_{13}$ corresponds to irreducible $\sigma_{t}$-twisted $V_{\mathcal{L}}(\ell_{123},0)$-modules.
	Recall that for any $h\in\C$, $L_{\mathcal{L}_{t}}(\ell_{1}, \ell_{3}, h)$
	is an irreducible restricted $\mathcal{L}_{t}$-module of level $\ell_{13}$, so 
	it is an irreducible $\sigma_{t}$-twisted $V_{\mathcal{L}}(\ell_{123},0)(\ell_{2}=0)$-module.
	
	Now we give the complete list of irreducible $\sigma_{t}$-twisted $V_{\mathcal{L}}(\ell_{123},0)$-modules.
	\begin{thm}
		Let $\ell_{2}=0$, $\ell_{1},\ell_{3}\in\mathbb{C}$. Then
		$\{L_{\mathcal{L}_{t}}(\ell_{1}, \ell_{3}, h)\ | \ h\in\C\}$ is a complete list of irreducible $\sigma_{t}$-twisted $V_{\mathcal{L}}(\ell_{123},0)$-modules.
	\end{thm}
	\begin{proof}
		Let $W=\coprod\limits_{r\in\C}W_{(r)}$ be an irreducible $\sigma_{t}$-twisted $V_{\mathcal{L}}(\ell_{123},0)$-module.
		By Theorem \ref{voamod}, $W$ is an irreducible restricted $\mathcal{L}_{t}$-module of level $\ell_{13}$. So $k_{i}$ acts on $W$ as a scalar $\ell_{i}$ for $i=1,3$.
		From Remark \ref{irrtwi}, there exists $h\in\C$ such that $W_{(h)}\neq 0$ and $W_{(h-n)}=0$ for all $n\in\frac{1}{t}\Z_{\geq 1}$.
		Let $0\neq w\in W_{(h)}$. Then 
		$$L_{0}w=hw,\;\;L_{m}w=0,\;\;
		I_{n+\frac{1}{t}}w=0$$
		for $m\geq 1, n\geq 0$.
		In view of Remark \ref{universal}, there is a unique $\mathcal{L}_{t}$-module homomorphism 
		$$\varphi:M_{\mathcal{L}_{t}}(\ell_{1}, \ell_{3}, h) \longrightarrow W$$ such that $\varphi({\bf 1}_{\Bbbk_{13},h})=w$. By Proposition \ref{hom},
		$\varphi$ is a $\sigma$-twisted $\langle U_{W}\rangle$-module homomorphism (since $\mathcal{L}_{t}$ generates the vertex algebra $\langle U_{W}\rangle$), where $\sigma$ is an order $t$ automorphism of the vertex algebra $\langle U_{W}\rangle$. 
		Recall that $M_{\mathcal{L}_{t}}(\ell_{1}, \ell_{3}, h)$ is a weak $\sigma_{t}$-twisted $V_{\mathcal{L}}(\ell_{123},0)$-module via the vertex algebra homomorphism $\psi$ in Theorem \ref{1}. So 
		$\varphi$ can be viewed as a $\sigma_{t}$-twisted $V_{\mathcal{L}}(\ell_{123},0)$-module homomorphism.
		Since $W$ is irreducible and $T_{\mathcal{L}_{t}}(\ell_{1}, \ell_{3}, h)$ is the (unique) largest proper submodule of $M_{\mathcal{L}_{t}}(\ell_{1}, \ell_{3}, h)$, it follows that
		$$\varphi(M_{\mathcal{L}_{t}}(\ell_{1}, \ell_{3}, h))=W$$
		and 
		$$\mbox{Ker}\;\varphi=T_{\mathcal{L}_{t}}(\ell_{1}, \ell_{3}, h).$$ Thus $\varphi$ reduces to a $\sigma_{t}$-twisted $V_{\mathcal{L}}(\ell_{123},0)$-module isomorphism from $L_{\mathcal{L}_{t}}(\ell_{1}, \ell_{3}, h)$ to $W$.
	\end{proof}
	
	It is interesting and important to classify the irreducible modules for the fixed point subalgebra $V_{\mathcal{L}}(\ell_{123},0)^{\sigma_{t}}:=\{v\in V_{\mathcal{L}}(\ell_{123},0) \ | \ \sigma_{t}(v)=v \}$, $t\in\Z_{\geq 1}$. In the case of $\ell_{2}=0$ and $\ell_{3}\neq 0$, we have $t=2$. Then $\sigma_{t}$ is the order 2 automorphism of $V_{\mathcal{L}}(\ell_{123},0)$ which is induced from its Heisenberg vertex operator subalgebra. Denote by $\sigma_{2}=\sigma$.
	Precisely, the automorphism
	$$\sigma: V_{\mathcal{L}}(\ell_{123},0)\longrightarrow V_{\mathcal{L}}(\ell_{123},0)$$
	is defined on the basis elements by
	$$I_{-k_{1}}\cdots I_{-k_{s}}L_{-m_{1}}\cdots L_{-m_{r}}{\bf 1}\mapsto (-1)^{s}I_{-k_{1}}\cdots I_{-k_{s}}L_{-m_{1}}\cdots L_{-m_{r}}{\bf 1},$$
	and extended linearly,
	where $r,s\geq 1$, $m_{1}\geq\cdots\geq m_{r}\geq 2$, $k_{1}\geq\cdots\geq k_{s}\geq 1$.
	Let
	$$V_{\mathcal{L}}(\ell_{123},0)^{+}=\{v\in V_{\mathcal{L}}(\ell_{123},0) \ | \ \sigma(v)=v\}$$ be the fixed point subalgebra under $\sigma$.

	For $\ell_{3}\neq 0$, denote by
	$$c_{\tilde{Vir}}=\ell_{1}-1+12\frac{\ell_{2}^{2}}{\ell_{3}}.$$
	Let $V_{\mathcal{H}}(\ell_{3},0)$ be the vertex operator algebra constructed from the Heisenberg subalgebra $\mathcal{H}$ which is equipped with the nonstandard conformal vector $\omega_{H}=\frac{1}{2\ell_{3}}I_{-1}I_{-1}{\bf 1}+\displaystyle\frac{\ell_{2}}{\ell_{3}}I_{-2}{\bf 1}$
	(of central charge $1-12\displaystyle\frac{\ell_{2}^{2}}{\ell_{3}}$).
	Let $\tilde{Vir}$ be the Virasoro algebra constructed by $\tilde{\omega}=\omega-\omega_{H}$.
	Let $V_{\tilde{Vir}}(c_{\tilde{Vir}},0)$ be the corresponding Virasoro vertex operator algebra. Recall that when $\ell_{3}\neq 0$, we have (cf. Theorem 3.16 of \cite{GW})
	$$V_{\mathcal{L}}(\ell_{123},0)\cong V_{\mathcal{H}}(\ell_{3},0)\otimes V_{\tilde{Vir}}(c_{\tilde{Vir}},0)$$
	as vertex operator algebras.

	It is well known that there exists an order 2 isomorphism of the Heisenberg vertex operator algebra $V_{\mathcal{H}}(\ell_{3},0)$.
	Let again
	$$\sigma: V_{\mathcal{H}}(\ell_{3},0)\longrightarrow V_{\mathcal{H}}(\ell_{3},0)$$
	be the order 2 automorphism
	which defined on the basis elements by 
	$$I_{-k_{1}}\cdots I_{-k_{s}}{\bf 1}\mapsto (-1)^{s}I_{-k_{1}}\cdots I_{-k_{s}}{\bf 1}$$
	and extended linearly,
	where $s\geq 1$, $k_{1}\geq\cdots\geq k_{s}\geq 1$.
	The fixed point subalgebra 
	$V_{\mathcal{H}}(\ell_{3},0)^{+}=\{v\in V_{\mathcal{H}}(\ell_{3},0) \ | \ \sigma(v)=v\}$
	has beed extensively studied (cf. \cite{DN} etc.).
	
	Then it is immediately to see that
	when $\ell_{3}\neq 0$ and $\ell_{2}=0$, 
	we
	have
	an isomorphism of vertex operator algebras
	$$V_{\mathcal{L}}(\ell_{123},0)^{+}\cong V_{\mathcal{H}}(\ell_{3},0)^{+}\otimes V_{\tilde{Vir}}(c_{\tilde{Vir}},0).$$
	Up to isomorphism, irreducible modules for the vertex operator algebra $V_{\mathcal{H}}(\ell_{3},0)^{+}$ are (cf. \cite{DN}, etc.)
	$V_{\mathcal{H}}(\ell_{3},0)^{\pm}, V_{\mathcal{H}}(\ell_{3},0)(\sigma)^{\pm},V_{\mathcal{H}}(\ell_{3},\lambda)\cong V_{\mathcal{H}}(\ell_{3},-\lambda),\;\;\forall\;0\neq \lambda\in\C.$
	Up to isomorphism, irreducible modules for the Virasoro vertex operator algebra $V_{\tilde{Vir}}(c_{\tilde{Vir}},0)$ are $L_{\tilde{Vir}}(c_{\tilde{Vir}},h)$ for all $h\in\C$ (cf. \cite{FZhu}, \cite{W}, etc.).
	Therefore, in the case of $\ell_{2}=0$ and $\ell_{3}\neq 0$,
	irreducible modules of $V_{\mathcal{L}}(\ell_{123},0)^{+}$
	are one-to-one correspond to the tensor product of irreducible modules of
	$V_{\mathcal{H}}(\ell_{3},0)^{+}$ and irreducible  modules of $V_{\tilde{Vir}}(c_{\tilde{Vir}},0)$ (cf. Proposition 4.7.2 and Theorem 4.7.4 of \cite{FHL}).

	For other automorphism $\sigma_{t}$, 
	the complete list of irreducible $V_{\mathcal{L}}(\ell_{123},0)^{\sigma_{t}}$-modules remains to be investigated.


\begin{thebibliography}{99}
		\bibitem[1]{ACKP} E. Arbarello, C. De Concini, V. G. Kac, C. Procesi, Moduli spaces of curves and representation theory,
		{\em Comm. Math. Phys.} {\bf 117} (1988), no. 1, 1-36.
		
		\bibitem[2]{AR1}
		D. Adamović, G. Radobolja, Free ﬁeld realization of the twisted Heisenberg– Virasoro algebra at level zero and its applications, {\em J. Pure Appl. Algebra} {\bf219} (2015), no. 10, 4322–4342. 
		
		\bibitem[3]{AR2}
		D. Adamović, G. Radobolja,
		Self-dual and logarithmic representations of the twisted Heisenberg-Virasoro algebra at level zero,
		{\em Commun. Contemp. Math.} {\bf21} (2019),  no. 2, 1850008, 26 pp.
		
		
		\bibitem[4]{B1} Y. Billig, 
		Representations of the twisted Heisenberg-Virasoro algebra at level zero,
		{\em Canad. Math. Bull.} {\bf 46}(2003), no. 4, 529-537.
		
		\bibitem[5]{B2} Y. Billig,
		A category of modules for the full toroidal Lie algebra,
		{\em Int. Math. Res. Not.} (2006), Art. ID. 68395, 46pp.
		
		\bibitem[6]{DN}
		C. Dong, K. Nagatomo, 
		Classification of irreducible modules for the vertex operator algebra $M(1)^+$,
		{\em J. Algebra} {\bf 216} (1999), no. 1, 384–404.
		
		\bibitem[7]{FHL}
		I. B. Frenkel, Y. Z. Huang, J. Lepowsky, 
		On axiomatic approaches to vertex operator algebras and modules,
		{\em Mem. Amer. Math. Soc.} {\bf104} (1993), no. 494, Viii+64 pp.
		
		\bibitem[8]{FZhu}
		I. B. Frenkel, Y. C. Zhu,
		Vertex operator algebras associated to representations of affine and Virasoro algebras,
		{\em Duke Math. J.} {\bf 66} (1992), no. 1, 123-168.
		
		\bibitem[9]{FZ} 
		I. B. Frenkel, A. M. Zeitlin, 
		Quantum group $GL_{q}(2)$
		and quantum Laplace operator via semi-infinite cohomology,
		{\em J. Noncommut. Geom.} (7) (2013), no. 4, 1007–1026. 	
		
		\bibitem[10]{GLTW}
		H. Guo, H. S. Li, S. Tan, Q. Wang, $q$-Virasoro algebra and vertex algebras, {\em J. Pure Appl. Algebra } {\bf 219} (2015), no. 4, 1258-1277.
		
		\bibitem[11]{GW1}
		H. Guo, Q. Wang, Associating vertex algebras with the unitary Lie algebra, {\em J. Algebra } {\bf 424} (2015), 126-146.
		
		\bibitem[12]{GW}
		H. Guo, Q. Wang, Twisted Heisenberg-Virasoro vertex operator algebra, {\em Glas. Mat. Ser. III} {\bf 54} (74) (2019),  no. 2, 369–407. 
		
		
		\bibitem[13]{L1}
		H. S. Li,  
		Local systems of twisted vertex operators, vertex operator superalgebras and twisted modules.
		In: Moonshine, the Monster, and related topics, pp. 203–236, 
		{\em Contemp. Math.} {\bf193}, Amer. Math. Soc., Providence, RI, 1996.
		
		\bibitem[14] {LL} 
		J. Lepowsky, H. S. Li,
		Introduction to Vertex Operator Algebras
		and Their Representations, {\em Progress in Mathematics}, Vol. 227,
		Birkh\"auser, Boston, 2004.
		
		\bibitem[15]{SJ}
		R. Shen, C. Jiang,
		The derivation algebra and automorphism group of the twisted Heisenberg-Virasoro algebra,
		{\em Comm. Algebra} {\bf 34} (2006), no. 7, 2547–2558.
		
		
		\bibitem[16]{W} W. Wang, 
		Rationality of Virasoro Vertex Operator Algebras,
		{\em Int. Math. Res. Not.} (1993), no. 7, 197-211.
		
	\end{thebibliography}
\end{document}